\documentclass[12pt]{amsart}
\usepackage{amsmath,amssymb,cite,geometry,bm,easybmat,enumerate}
\usepackage[colorlinks,citecolor=blue,linkcolor=blue,urlcolor=blue]{hyperref}
\newtheorem{theorem}{Theorem}[section]
\newtheorem{lemma}[theorem]{Lemma}
\newtheorem{proposition}[theorem]{Proposition}

\newtheorem{remark}[theorem]{Remark}
\numberwithin{equation}{section}
\allowdisplaybreaks[4]

\begin{document}

\title[Determining Lam\'{e} coefficients by elastic Dirichlet-to-Neumann map]{Determining Lam\'{e} coefficients by elastic Dirichlet-to-Neumann map on a Riemannian manifold}

\author{Xiaoming Tan}
\address[Xiaoming Tan]{School of Mathematics and Statistics, Beijing Institute of Technology, Beijing 100081, China}
\email{xtan@bit.edu.cn}

\author{Genqian Liu}
\address[Genqian Liu]{School of Mathematics and Statistics, Beijing Institute of Technology, Beijing 100081, China}
\email{liugqz@bit.edu.cn}


\keywords{Lam\'{e} system; Elastic Calder\'{o}n's problem; Inverse problems; Elastic Dirichlet-to-Neumann map; Pseudodifferential operators\\
{\bf ---------------}\\
{\scriptsize Xiaoming Tan\\
{\it School of Mathematics and Statistics, Beijing Institute of Technology, Beijing 100081, China}\\
{\it Email address}: xtan@bit.edu.cn}\\[1mm]
Genqian Liu\\
{\it School of Mathematics and Statistics, Beijing Institute of Technology, Beijing 100081, China}\\
{\it Email address}: liugqz@bit.edu.cn}

\begin{abstract}
    For the Lam\'{e} operator $\mathcal{L}_{\lambda,\mu}$ with variable coefficients $\lambda$ and $\mu$ on a smooth compact Riemannian manifold $(M,g)$ with smooth boundary $\partial M$, we give an explicit expression for full symbol of the elastic Dirichlet-to-Neumann map $\Lambda_{\lambda,\mu}$. We show that $\Lambda_{\lambda,\mu}$ uniquely determines partial derivatives of all orders of the Lam\'{e} coefficients $\lambda$ and $\mu$ on $\partial M$. Moreover, for a nonempty open subset $\Gamma\subset\partial M$, suppose that the manifold and the Lam\'{e} coefficients are real analytic up to $\Gamma$, we prove that $\Lambda_{\lambda,\mu}$ uniquely determines the Lam\'{e} coefficients on the whole manifold $\bar{M}$.
\end{abstract}

\maketitle 

\section{Introduction}

\addvspace{5mm}

Let $(M,g)$ be a smooth compact Riemannian manifold of dimension $n$ with smooth boundary $\partial M$. In this paper, we consider $M$ as an inhomogeneous, isotropic, elastic medium. Assume that the Lam\'{e} coefficients $\lambda,\,\mu \in C^{\infty}(\bar{M})$ satisfy
\begin{align}
    \mu>0, \quad \lambda + \mu \geqslant 0.
\end{align}

In the local coordinates $\{x_j\}_{j=1}^n$, we denote by $\bigl\{\frac{\partial}{\partial x_j}\bigr\}_{j=1}^n$ and $\{dx_j\}_{j=1}^n$, respectively, the natural basis for the tangent space $T_x M$ and the cotangent space $T_x^{*} M$ at the point $x\in M$. In what follows, we will use the Einstein summation convention. The Greek indices run from 1 to $n-1$, whereas the Roman indices run from 1 to $n$, unless otherwise specified. Then the Riemannian metric $g$ is given by $g = g_{jk} \,dx_j\otimes dx_k$. 

Let $\nabla_j=\nabla_{\frac{\partial}{\partial x_j}}$ be the covariant derivative with respect to $\frac{\partial}{\partial x_j}$ and $\nabla^j= g^{jk} \nabla_k$. Then for a displacement vector field $\textbf{\textit{u}}=u^j\frac{\partial}{\partial x_j}$ and a tensor field $\mathcal{T} =\mathcal{T}^j_k\,dx_k\otimes \frac{\partial}{\partial x_j}$ of type $(1,1)$, we denote by $\operatorname{div}$ the divergence operator, i.e.,
\begin{align}\label{0.01}
    \operatorname{div}\textbf{\textit{u}}:=\nabla_j u^j,\quad
    (\operatorname{div}\mathcal{T})^j:=\nabla^k \mathcal{T}^j_k.
\end{align}
The gradient operator is given by
\begin{align}\label{0.02}
    (\operatorname{grad}v)^j:=\nabla^j v,\quad v \in C^{\infty}(M).
\end{align}
We define the Cauchy stress tensor $\tau$ of type $(1,1)$ as follows:
\begin{align}\label{0.1}
    \tau:=\lambda (\operatorname{div} \textbf{\textit{u}}) I_n + \mu (S\textbf{\textit{u}}),
\end{align}
or $\tau^j_k=\lambda (\operatorname{div} \textbf{\textit{u}}) \delta^j_k + \mu (S\textbf{\textit{u}})^j_k$. Here $I_n$ is the $n\times n$ identity matrix, $\delta^j_k$ is the Kronecker delta and
\begin{align}\label{1.5}
    (S\textbf{\textit{u}})^j_k := \nabla^j u_k + \nabla_k u^j,
\end{align}
where $u_k=g_{kl}u^l$.

\addvspace{3mm}

The Lam\'{e} operator $\mathcal{L}_{\lambda,\mu}$ with variable coefficients $\lambda$ and $\mu$ on the Riemannian manifold is given by
\begin{align}\label{0.2}
    \mathcal{L}_{\lambda,\mu} \textbf{\textit{u}}
    := \operatorname{div} \tau 
    =\operatorname{grad}(\lambda \operatorname{div} \textbf{\textit{u}})+ \operatorname{div} \bigl(\mu (S\textbf{\textit{u}})\bigr)
\end{align}
for the displacement $\textbf{\textit{u}}$. Note that, by \eqref{0.01}, \eqref{0.02} and \eqref{1.5} we have
\begin{align*}
    \operatorname{div} \bigl(\mu (S\textbf{\textit{u}})\bigr) 
    &= \nabla^k\bigl(\mu(S\textbf{\textit{u}})^j_k\bigr)
    =(\nabla^k\mu)(S\textbf{\textit{u}})^j_k+\mu\nabla^k(S\textbf{\textit{u}})^j_k,\\
    \nabla^k\nabla^j u_k 
    &= \nabla^j\nabla^k u_k + g^{jk}R_{kl} u^l
    =\nabla^j(\operatorname{div} \textbf{\textit{u}}) + \operatorname{Ric}(\textbf{\textit{u}})^j.
\end{align*}
Here $\operatorname{Ric}(\textbf{\textit{u}})^j = g^{jk}R_{kl} u^l$, where $R_{kl}$ are the components of Ricci tensor of the manifold, i.e.,
\begin{align}\label{0.03}
    R_{kl}=\frac{\partial \Gamma^{j}_{kl}}{\partial x_j} - \frac{\partial \Gamma^{j}_{jl}}{\partial x_k} + \Gamma^{j}_{jm} \Gamma^{m}_{kl} - \Gamma^{j}_{km} \Gamma^{m}_{jl},
\end{align}
and $\Gamma^{j}_{kl} = \frac{1}{2} g^{jm} \bigl(\frac{\partial g_{km}}{\partial x_l} + \frac{\partial g_{lm}}{\partial x_k} - \frac{\partial g_{kl}}{\partial x_m}\bigr)$. Thus, \eqref{0.2} can be written as
\begin{align}\label{1.1}
    \mathcal{L}_{\lambda,\mu} \textbf{\textit{u}} 
    &:= \mu \Delta^{}_{B} \textbf{\textit{u}} + (\lambda + \mu)\operatorname{grad}\operatorname{div} \textbf{\textit{u}} + \mu \operatorname{Ric}(\textbf{\textit{u}}) \\
    &\qquad + (\operatorname{grad} \lambda) \operatorname{div} \textbf{\textit{u}} + (S\textbf{\textit{u}})(\operatorname{grad} \mu), \notag
\end{align} 
where the Bochner Laplacian $(\Delta^{}_{B}\textbf{\textit{u}})^j := \nabla^k \nabla_k u^j$. 

In particular, for a bounded Euclidean domain, the Lam\'{e} operator with constant Lam\'{e} coefficients has the form $L\textbf{\textit{u}}:=\mu\Delta\textbf{\textit{u}}+(\lambda+\mu)\nabla(\nabla\cdot\textbf{\textit{u}})$ (see \cite{Kupr80,LandLif86}).

By a direct computation, we have (see \cite{Liu19})
\begin{align*}
    (\Delta^{}_{B}\textbf{\textit{u}})^j 
    & = \Delta_g u^j + g^{kl} 
    \Bigl(
        2\Gamma^j_{mk}\frac{\partial u^m}{\partial x_l} + \frac{\partial \Gamma^j_{mk}}{\partial x_l}u^m + (\Gamma^j_{hl}\Gamma^h_{mk} - \Gamma^h_{kl}\Gamma^j_{mh}) u^m
    \Bigr),
\end{align*}
where the Laplace--Beltrami operator is given by
\begin{equation}\label{2.7}
    \Delta_{g} v
    = g^{jk}\Bigl(\frac{\partial^2 v}{\partial x_j \partial x_k} - \Gamma^l_{jk}\frac{\partial v}{\partial x_l}\Bigr), \quad v \in C^{\infty}(M).
\end{equation}
It follows from \eqref{0.03} that
\begin{align*}
    (\Delta^{}_{B}\textbf{\textit{u}})^j=\Delta_g u^j - \operatorname{Ric}(\textbf{\textit{u}})^j + g^{kl} \Bigl( 2\Gamma^j_{mk} \frac{\partial u^m}{\partial x_l} + \frac{\partial \Gamma^j_{kl}}{\partial x_m} u^m \Bigr).
\end{align*}
Therefore, the Lam\'{e} operator \eqref{1.1} with variable coefficients $\lambda$ and $\mu$ can be rewritten as
\begin{align}\label{1.6}
    (\mathcal{L}_{\lambda,\mu} \textbf{\textit{u}})^j 
    &= \mu \Delta_{g}u^j + (\lambda + \mu)\nabla^j \nabla_k u^k + (\nabla^j \lambda) \nabla_k u^k + (\nabla^k \mu)(\nabla_k u^j + \nabla^j u_k) \\
    &\quad + \mu g^{kl} \Bigl( 2\Gamma^j_{km} \frac{\partial u^m}{\partial x_l} + \frac{\partial \Gamma^j_{kl}}{\partial x_m} u^m \Bigr),\quad j=1,2,\dots,n.\notag
\end{align}

\addvspace{3mm}

We now consider the following Dirichlet boundary value problem for the Lam\'{e} system on the Riemannian manifold:
\begin{align}\label{1.3}
    \begin{cases}
        \mathcal{L}_{\lambda,\mu} \textbf{\textit{u}} = 0  & \text{in}\ M, \\
        \textbf{\textit{u}}= \textbf{\textit{f}} \quad & \text{on}\ \partial M.
    \end{cases}
\end{align}
For any displacement $\textbf{\textit{f}}\in[H^{1/2}(\partial M)]^n$ on the boundary, by the theory of elliptic operators there is a unique solution $\textbf{\textit{u}}$ solves the above problem \eqref{1.3}. Therefore, the elastic Dirichlet-to-Neumann map (also called displacement-to-traction map) $\Lambda_{\lambda,\mu}:[H^{1/2}(\partial M)]^n \to [H^{-1/2}(\partial M)]^n$ associated with the operator $\mathcal{L}_{\lambda,\mu}$ is defined by
\begin{align}\label{1.8}
    \Lambda_{\lambda,\mu}(\textbf{\textit{f}}) := \tau(\nu) = \lambda (\operatorname{div} \textbf{\textit{u}})\nu + \mu (S \textbf{\textit{u}})\nu \quad \text{on}\ \partial M,
\end{align}
where $\nu$ is the outward unit normal vector to the boundary $\partial M$. The elastic Dirichlet-to-Neumann map $\Lambda_{\lambda,\mu}$ is an elliptic, self-adjoint pseudodifferential operator of order one defined on the boundary $\partial M$. In \cite{Liu19}, the second author of this paper gave the explicit expression for the Lam\'{e} operator with constant coefficients and obtained the heat trace asymptotic expansion associated with the elastic Dirichlet-to-Neumann map on a Riemannian manifold. In this paper, by giving an explicit expression for the elastic Dirichlet-to-Neumann map $\Lambda_{\lambda,\mu}$ with variable coefficients, we will study the elastic Calder\'{o}n problem on a Riemannian manifold that is determining variable Lam\'{e} coefficients by the elastic Dirichlet-to-Neumann map $\Lambda_{\lambda,\mu}$.

\addvspace{3mm}

We briefly recall the classical Calder\'{o}n problem \cite{Cald80}: whether one can uniquely determine the electrical conductivity of a medium by making voltage and current measurements at the boundary of the medium? To be more precise, let $\Omega\subset\mathbb{R}^n$ be a smooth bounded domain with smooth boundary and let a bounded positive function $\gamma$ be the electrical conductivity. In the absence of sinks or sources of current, the Dirichlet problem for the conductivity equation is
\begin{align*}
    \begin{cases}
        \nabla\cdot(\gamma\nabla u)=0  & \text{in}\ \Omega, \\
        u= f \quad & \text{on}\ \partial \Omega.
    \end{cases}
\end{align*}
The classical Dirichlet-to-Neumann map (also called voltage-to-current map) is defined by
\begin{align*}
    \Lambda_\gamma(f):=\gamma\frac{\partial u}{\partial \nu}\Big|_{\partial \Omega}.
\end{align*}
The classical Calder\'{o}n problem is to determine the conductivity $\gamma$ by the classical Dirichlet-to-Neumann map $\Lambda_\gamma$. This problem has been studied for decades. In dimensions $n\geqslant 2$, Kohn and Vogelius \cite{KohnVoge84} proved the famous uniqueness result on the boundary for $C^{\infty}$-conductivities, that is, if $\Lambda_{\gamma_1}=\Lambda_{\gamma_2}$, then 
\begin{align*}
    \frac{\partial^{|J|} \gamma_1}{\partial x^J}\Big|_{\partial \Omega}=\frac{\partial^{|J|} \gamma_2}{\partial x^J}\Big|_{\partial \Omega}
\end{align*}
for all multi-indices $J=(J_1,J_2,\dots,J_n)\in\mathbb{N}^n$. This settled the uniqueness problem on the boundary in the real analytic category. They extended the uniqueness result to piecewise real analytic conductivities in \cite{KohnVoge85}. In dimensions $n\geqslant 3$, in the celebrated paper \cite{SylvUhlm87} Sylvester and Uhlmann proved the uniqueness of the $C^{\infty}$-conductivities in a bounded domain $\Omega\subset\mathbb{R}^n$ with smooth boundary by constructing the complex geometrical optics solutions. Haberman and Tataru \cite{Haberata13} proved uniqueness for $C^1$-conductivities and Lipschitz conductivities sufficiently close to the identity. Caro and Rogers \cite{CaroRogers16} reduced regularities of conductivities and boundary to Lipschitz. In two dimensional case, many important results were also obtained. Nachmann \cite{Nachman96} gave the corresponding uniqueness result for $W^{2,p}$-conductivities $(p>1)$. This uniqueness result be extended significantly by Astala and P\"{a}iv\"{a}rinta \cite{AstPai06} for $L^{\infty}$-conductivities. The classical Calder\'{o}n problem have attracted lots of attention for decades (see \cite{KohnVoge84,KohnVoge85,Nachman96,SylvUhlm87,ImaUhlmYama10,ALP05,SunUhl03,SunUhl97,LeeUhlm89} and references therein). We also refer the reader to the survey articles \cite{Uhlm09,Uhlm14} for the classical Calder\'{o}n problem and related topics.

\addvspace{3mm}

Let us come back to the elastic Calder\'{o}n problem. Partial uniqueness results for the determination of Lam\'{e} coefficients from boundary measurements were obtained. For a bounded Euclidean domain $\Omega\subset\mathbb{R}^n$ with smooth boundary $\partial\Omega$, Nakamura and Uhlmann \cite{NakaUhlm95} proved that one can determine the full Taylor series of the Lam\'{e} coefficients on the boundary in all dimensions $n\geqslant 2$ and for a generic anisotropic elastic tensor in two dimensions. For a bounded two dimensional Euclidean domain, Akamatsu, Nakamura and Steinberg \cite{ANS91} gave an inversion formula for the normal derivatives on the boundary of the Lam\'{e} coefficients from the elastic Dirichlet-to-Neumann map. Nakamura and Uhlmann \cite{NakaUhlm93} proved that the elastic Dirichlet-to-Neumann map uniquely determines the Lam\'{e} coefficients $\lambda,\mu \in W^{31,\infty}(\bar{\Omega})$, provided that $\lambda,\mu$ are sufficiently close to positive constants. Imanuvilov and Yamamoto \cite{ImaYama11} proved the uniqueness by the elastic Dirichlet-to-Neumann map limited to an arbitrary sub-boundary, provided that $\lambda\in C^3(\bar{\Omega})$ and $\mu$ is a constant. Imanuvilov and Yamamoto \cite{ImaYama15} also proved the global uniqueness of the Lam\'{e} coefficients $\lambda,\mu \in C^{10}(\bar{\Omega})$. In three dimensional Euclidean domains, Nakamura and Uhlmann \cite{NakaUhlm94,NakaUhlm03} as well as Eskin and Ralston \cite{EskinRalston02} proved the global uniqueness of the Lam\'{e} coefficients provided that $\nabla\mu$ is small in a suitable norm. Imanuvilov, Uhlmann and Yamamoto \cite{ImaUhlmYama12} proved that the Lam\'{e} coefficients are uniquely recovered from partial Cauchy data, provided that $\lambda_1=\lambda_2$ on an arbitrary open subset $\Gamma_0\subset\partial\Omega$ and $\mu_1,\mu_2$ are some constants. However, in dimensions $n\geqslant 3$, the global uniqueness of the Lam\'{e} coefficients $\lambda,\mu\in C^{\infty}(\bar{\Omega})$ without the smallness assumption ($\|\nabla\mu\| < \varepsilon_0$ for some small positive $\varepsilon_0$) remains an open problem (see \cite[p.\,210]{Isakov17}).

Various inverse problems occurring in mathematics, physics and engineering have been studied for decades. Here we do not list all the references about these topics. We refer the reader to \cite{Liu19.2,Pichler18,CaroZhou14,JoshiMcDowall00,McDowall97,OlaPai93} for Maxwell's equations, to \cite{Liu20,HeckWangLi07,LiWang07,DKSU09,DKSU07,KrupUhlm18,KrupUhlm14,NSU95,PSU10} for incompressible fluid and many others. For the studies about other types of Dirichlet-to-Neumann map, we also refer the reader to \cite{Liu11,Liu15,Liu14,LiuTan21,LiuTan22.2} and references therein.

\addvspace{3mm}

For the sake of simplicity, we denote by
\begin{align*}
    \begin{bmatrix}
        [a^\alpha_\beta] & [a^{\alpha}] \\[1mm]
        [a_{\beta}] & a^{n}_{n}
    \end{bmatrix}
    :=\begin{bmatrix}
        \begin{BMAT}{ccc.c}{ccc.c}
            a^1_{1} & \dots & a^1_{n-1} & a^1_{n} \\
            \vdots & \ddots & \vdots & \vdots \\
            a^{n-1}_1 & \dots & a^{n-1}_{n-1} & a^{n-1}_{n}\\
            a^{n}_1 & \dots & a^{n}_{n-1} & a^{n}_{n}
        \end{BMAT}
    \end{bmatrix}.
\end{align*}

\addvspace{5mm}

The main results of this paper are the following three theorems.
\begin{theorem}\label{thm1.3}
    Let $(M,g)$ be a smooth compact Riemannian manifold of dimension $n$ with smooth boundary $\partial M$. Assume that the Lam\'{e} coefficients $\lambda,\,\mu \in C^{\infty}(\bar{M})$ satisfy $\mu>0$ and $\lambda + \mu \geqslant 0$. Let $\sigma(\Lambda_{\lambda,\mu}) \sim \sum_{j\leqslant 1} p_j(x,\xi^{\prime})$ be the full symbol of the elastic Dirichlet-to-Neumann map $\Lambda_{\lambda,\mu}$. Then
    \begin{align}
        \label{18} p_1&=
        \begin{bmatrix}
            \mu|\xi^{\prime}|I_{n-1} + \frac{\mu(\lambda+\mu)}{(\lambda+3\mu)|\xi^{\prime}|} [\xi^\alpha\xi_\beta] & -\frac{2i\mu^2}{\lambda+3\mu} [\xi^\alpha]\\[2mm]
            \frac{2i\mu^2}{\lambda+3\mu}[\xi_\beta] & \frac{2\mu(\lambda+2\mu)}{\lambda+3\mu}|\xi^{\prime}|
        \end{bmatrix},\\[2mm]
        \label{19} p_0&=
        \begin{bmatrix}
            \mu I_{n-1} &0  \\
            0& \lambda+2\mu 
        \end{bmatrix}
        q_0-
        \begin{bmatrix}
            0 & 0 \\
            \lambda [\Gamma^\alpha_{\alpha\beta}] & \lambda \Gamma^\alpha_{\alpha n} 
        \end{bmatrix},\\[2mm]
        \label{20} p_{-m}&=
        \begin{bmatrix}
            \mu I_{n-1} &0  \\
            0& \lambda+2\mu 
        \end{bmatrix}
        q_{-m},\quad m\geqslant 1,
    \end{align}
    where $i=\sqrt{-1}$, $\xi^{\prime}=(\xi_1,\dots,\xi_{n-1})$, $\xi^\alpha=g^{\alpha\beta}\xi_\beta$, $|\xi^{\prime}|=\sqrt{\xi^\alpha\xi_\alpha}$, and $q_{-m}\,(m\geqslant 0)$ are given by \eqref{3.1.1} in Section $\ref{s2}$.
\end{theorem}

\addvspace{3mm}

For the case of constant Lam\'{e} coefficients, the full symbol of the elastic Dirichlet-to-Neumann map on a Riemannian manifold had been obtained by the second author of this paper in \cite{Liu19}. The principal symbol of the elastic Dirichlet-to-Neumann map had also be studied in \cite{Vodev22} and \cite{ZhangY20} in the context of the elastic wave equations in Euclidean setting.

By studying the full symbol of the elastic Dirichlet-to-Neumann map $\Lambda_{\lambda,\mu}$, we prove the following result:

\begin{theorem}\label{thm1.1}
    Let $(M,g)$ be a smooth compact Riemannian manifold of dimension $n$ with smooth boundary $\partial M$. Assume that the Lam\'{e} coefficients $\lambda,\,\mu \in C^{\infty}(\bar{M})$ satisfy $\mu>0$ and $\lambda + \mu \geqslant 0$. Then the elastic Dirichlet-to-Neumann map $\Lambda_{\lambda,\mu}$ uniquely determines $\frac{\partial^{|J|} \lambda}{\partial x^J}$ and $\frac{\partial^{|J|} \mu}{\partial x^J}$ on the boundary for all multi-indices $J$.
\end{theorem}

The uniqueness result in Theorem \ref{thm1.1} can be extended to the whole manifold for real analytic setting.

\begin{theorem}\label{thm1.2}
    Let $(M,g)$ be a compact Riemannian manifold of dimension $n$ with smooth boundary $\partial M$, and let $\Gamma\subset \partial M$ be a nonempty open subset. Suppose that the manifold $M$ is real analytic up to $\Gamma$, assume that the Lam\'{e} coefficients $\lambda,\,\mu$ are also real analytic up to $\Gamma$ and satisfy $\mu>0$ and $\lambda + \mu \geqslant 0$. Then the elastic Dirichlet-to-Neumann map $\Lambda_{\lambda,\mu}$ uniquely determines the Lam\'{e} coefficients $\lambda$ and $\mu$ on $\bar{M}$.
\end{theorem}

\addvspace{3mm}

Theorem \ref{thm1.2} shows that the global uniqueness of real analytic Lam\'{e} coefficients on a real analytic Riemannian manifold. To the best of our knowledge, this is the first global uniqueness result for Lam\'{e} coefficients on a Riemannian manifold. It is clear that Theorem {\rm\ref{thm1.2}} also holds for a real analytic Euclidean domain.

\addvspace{3mm}

The main ideas of this paper are as follows. Firstly, in \cite{Liu19} the second author of this paper established a method such that one can calculate the full symbol of the elastic Dirichlet-to-Neumann map with constant coefficients. By this method we can deal with the case for $\lambda,\,\mu \in C^{\infty}(\bar{M})$. Then we flatten the boundary and induce a Riemannian metric in a neighborhood of the boundary and give a local representation for the elastic Dirichlet-to-Neumann map $\Lambda_{\lambda,\mu}$ with variable coefficients in boundary normal coordinates, that is,
\begin{align*}
    \Lambda_{\lambda,\mu} = A\Big(-\frac{\partial }{\partial x_n}\Big)-D,
\end{align*}
where $A$ and $D$ are two matrices. We then look for the following factorization for the Lam\'{e} operator, we get
\begin{align*}
    A^{-1} \mathcal{L}_{\lambda,\mu} 
    = I_{n}\frac{\partial^2 }{\partial x_n^2} + B \frac{\partial }{\partial x_n} + C
    = \Bigl(I_{n}\frac{\partial }{\partial x_n} + B - Q\Bigr)\Bigl(I_{n}\frac{\partial }{\partial x_n} + Q\Bigr),
\end{align*}
where $B$, $C$ are two differential operators and $Q$ is a pseudodifferential operator. As a result, we obtain the equation
\begin{align*}
    Q^2 - BQ - \Bigl[\frac{\partial }{\partial x_n},Q\Bigr] + C = 0,
\end{align*}
where $[\cdot,\cdot]$ is the commutator. Finally, we solve the full symbol equation
\begin{align*}
    \sum_{J^{\prime}} \frac{(-i)^{|J^{\prime}|}}{J^{\prime} !} \partial_{\xi^{\prime}}^{J^{\prime}}\!q \, \partial_{x^\prime}^{J^{\prime}}\!q - \sum_{J^{\prime}} \frac{(-i)^{|J^{\prime}|}}{J^{\prime} !} \partial_{\xi^{\prime}}^{J^{\prime}}\!b \, \partial_{x^\prime}^{J^{\prime}}\!q - \frac{\partial q}{\partial x_n} + c = 0,
\end{align*}
which is a matrix equation, where the sum is over all multi-indices $J^{\prime}$, $\xi^{\prime}=(\xi_1,\dots,\xi_{n-1})$ and $x^\prime=(x_1,\dots,x_{n-1})$. Here $b$, $c$ and $q$ are the full symbols of the operators $B$, $C$ and $Q$, respectively. Thus, we obtain the full symbol $\sigma(\Lambda_{\lambda,\mu}) \sim \sum_{j\leqslant 1} p_j(x,\xi^{\prime})$ of $\Lambda_{\lambda,\mu}$ from the full symbol of $Q$. Note that computations of the full symbols of matrix-valued pseudodifferential operators are quite difficult tasks. Generally, the above full symbol equation can not be exactly solved, in other words, there is not a general formula of the solution represented by the coefficients of the matrix equation. Hence, by overcoming the difficulties of computing the symbols of pseudodifferential operators and solving the symbol equation with variable coefficients $\lambda,\,\mu \in C^{\infty}(\bar{M})$, we develop the method of the previous work \cite{Liu19} to deal with the uniqueness of the Lam\'{e} coefficients on the Riemannian manifold. The symbols $p_j(x,\xi^{\prime})$ contain the information about the Lam\'{e} coefficients $\lambda,\,\mu$ and their derivatives on the boundary, thus we can prove that the elastic Dirichlet-to-Neumann map $\Lambda_{\lambda,\mu}$ uniquely determines the Lam\'{e} coefficients on the boundary. Furthermore, we prove that the Lam\'{e} coefficients can be uniquely determined on the whole manifold $\bar{M}$ by the elastic Dirichlet-to-Neumann map $\Lambda_{\lambda,\mu}$, provided the manifold and the Lam\'{e} coefficients are real analytic.

This paper is organized as follows. In Section \ref{s2} we give the explicit expression of the elastic Dirichlet-to-Neumann map $\Lambda_{\lambda,\mu}$ in boundary normal coordinates and derive a factorization of the Lam\'{e} operator $\mathcal{L}_{\lambda,\mu}$ with variable coefficients, then we compute the full symbols of $\Lambda_{\lambda,\mu}$ and the pseudodifferential operator $Q$. In Section \ref{s3} we prove Theorem \ref{thm1.3} and Theorem \ref{thm1.1} for boundary determination. Finally, Section \ref{s4} is devoted to proving Theorem \ref{thm1.2} for global uniqueness.

\addvspace{10mm}

\section{Symbols of the pseudodifferential operators}\label{s2}

\addvspace{5mm}

Here we briefly introduce the construction of geodesic coordinates with respect to the boundary $\partial M$ (see \cite{LeeUhlm89} or \cite[p.\,532]{Taylor11.2}). For each boundary point $x^{\prime} \in \partial M$, let $\gamma_{x^{\prime}}:[0,\varepsilon)\to \bar{M}$ denote the unit-speed geodesic starting at $x^{\prime}$ and normal to $\partial M$. If $ x^{\prime} := \{x_{1}, \ldots, x_{n-1}\}$ are any local coordinates for $\partial M$ near $x_0 \in \partial M$, we can extend them smoothly to functions on a neighborhood of $x_0$ in $\bar{M}$ by letting them be constant along each normal geodesic $\gamma_{x^{\prime}}$. If we then define $x_n$ to be the parameter along each $\gamma_{x^{\prime}}$, it follows easily that $\{x_{1}, \ldots, x_{n}\}$ form coordinates for $\bar{M}$ in some neighborhood of $x_0$, which we call the boundary normal coordinates determined by $\{x_{1}, \ldots, x_{n-1}\}$. In these coordinates $x_n>0$ in $M$, and $\partial M$ is locally characterized by $x_n=0$. A standard computation shows that the metric has the form $g = g_{\alpha\beta} \,dx_{\alpha} \,dx_{\beta} + dx_{n}^{2}$.
\begin{proposition}
    In boundary normal coordinates, the elastic Dirichlet-to-Neumann map $\Lambda_{\lambda,\mu}$ can be written as
    \begin{align}\label{3.1}
        \Lambda_{\lambda,\mu} = A\Bigl(-\frac{\partial }{\partial x_n}\Bigr)-D,
    \end{align}
    where
    \begin{align}
        \label{3.2}
        A&=
        \begin{bmatrix}
            \mu I_{n-1} &0  \\
            0& \lambda+2\mu 
        \end{bmatrix},\\
        \label{3.21}
        D&=
        \begin{bmatrix}
            0 & \mu \bigl[g^{\alpha\beta}\frac{\partial }{\partial x_\beta}\bigr] \\[1mm]
            \lambda \bigl[\frac{\partial }{\partial x_\beta} + \Gamma^\alpha_{\alpha\beta} \bigr] & \lambda \Gamma^\alpha_{\alpha n}
        \end{bmatrix}.
    \end{align}
\end{proposition}

\addvspace{5mm}

\begin{proof}
    By \eqref{1.5}, we have
    \begin{align*}
        ((S\textbf{\textit{u}})\nu)^j =(S\textbf{\textit{u}})^j_k \nu^k
        &=(\nabla^j u_k + \nabla_k u^j)\nu^k
    \end{align*}
    In boundary normal coordinates, we take $\nu=(0,\dots,0,-1)^t$ and $\partial_\nu = -\partial_{x_n}$. In particular, $u_n=u^n$ since $g_{jn}=\delta_{jn}$ in boundary normal coordinates. We get
    \begin{align*}
        ((S\textbf{\textit{u}})\nu)^j = -(\nabla^j u_n + \nabla_n u^j).
    \end{align*}
    Note that $\Gamma^n_{nk}=\Gamma^k_{nn}=0$ and $g^{\alpha\beta}\Gamma^n_{\beta\gamma}+\Gamma^\alpha_{n\gamma}=0$ in boundary normal coordinates. Thus
    \begin{align}
        ((S\textbf{\textit{u}})\nu)^\alpha \label{3.3}
        &=-(\nabla^\alpha u_n + \nabla_n u^\alpha)\\ 
        &=-\Bigl[g^{\alpha\beta}\Bigl(\frac{\partial u^n}{\partial x_\beta}+\Gamma^n_{\beta\gamma}u^\gamma\Bigr)+\frac{\partial u^\alpha}{\partial x_n}+\Gamma^\alpha_{n\gamma}u^\gamma\Bigr]\notag\\
        &=-g^{\alpha\beta}\frac{\partial u^n}{\partial x_\beta}-\frac{\partial u^\alpha}{\partial x_n},\notag\\
        ((S\textbf{\textit{u}})\nu)^n 
        &=-(\nabla^n u_n + \nabla_n u^n)=-2\frac{\partial u^n}{\partial x_n}.\label{3.4.1}
    \end{align}
    Hence, we immediately obtain \eqref{3.1} by combining \eqref{1.8}, \eqref{3.3} and \eqref{3.4.1}.
\end{proof}

\addvspace{3mm}

In boundary normal coordinates, we write the Laplace--Beltrami operator as
\begin{align}
    \Delta_g
    & = \frac{\partial^2 }{\partial x_n^2} + \Gamma^\alpha_{\alpha n} \frac{\partial }{\partial x_n} + g^{\alpha\beta} \frac{\partial^2}{\partial x_\alpha\partial x_\beta} +
    \Bigl(
        g^{\alpha\beta} \Gamma^\gamma_{\gamma\alpha} + \frac{\partial g^{\alpha\beta}}{\partial x_\alpha}
    \Bigr)
    \frac{\partial }{\partial x_\beta}.
\end{align}
According to \eqref{1.6}, we deduce that
\begin{align}\label{3.07}
    A^{-1} \mathcal{L}_{\lambda,\mu} = I_{n}\frac{\partial^2 }{\partial x_n^2} + B \frac{\partial }{\partial x_n} + C,
\end{align}
where $A$ is given by \eqref{3.2}, $B=B_1+B_0$, $C =C_2+C_1+C_0$, and
\begin{align*}
    B_1&=(\lambda+\mu)
    \begin{bmatrix}
        0 & \frac{1}{\mu}\bigl[g^{\alpha\beta}\frac{\partial}{\partial x_{\beta}}\bigr]\\
        \frac{1}{\lambda+2\mu}\bigl[\frac{\partial}{\partial x_{\beta}}\bigr] & 0 
    \end{bmatrix},\\[2mm]
    B_0&=
    \begin{bmatrix}
        \Gamma^\alpha_{\alpha n} I_{n-1}+2[\Gamma^\alpha_{\beta n}] & 0 \\[2mm]
        \frac{\lambda+\mu}{\lambda+2\mu}[\Gamma^\alpha_{\alpha\beta}] & \Gamma^\alpha_{\alpha n}
    \end{bmatrix}
    +\begin{bmatrix}
        \frac{1}{\mu} \frac{\partial \mu}{\partial x_{n}} I_{n-1} & \frac{1}{\mu} [\nabla^\alpha \lambda] \\[2mm]
        \frac{1}{\lambda+2\mu} \big[\frac{\partial \mu}{\partial x_{\beta}}\big] & \frac{1}{\lambda+2\mu} \frac{\partial (\lambda+2\mu)}{\partial x_{n}}
    \end{bmatrix},\\[2mm]
    C_2&=
    \begin{bmatrix}
        (g^{\alpha\beta}\frac{\partial^2 }{\partial x_\alpha \partial x_\beta})I_{n-1} + \frac{\lambda+\mu}{\mu}\bigl[g^{\alpha\gamma}\frac{\partial^2 }{\partial x_\gamma \partial x_\beta}\bigr] & 0 \\
        0 & \frac{\mu}{\lambda+2\mu}g^{\alpha\beta}\frac{\partial^2 }{\partial x_\alpha \partial x_\beta} 
    \end{bmatrix},\\[2mm]
    C_1&=
    \begin{bmatrix}
        \bigl((g^{\alpha\beta}\Gamma^\gamma_{\alpha\gamma}+\frac{\partial g^{\alpha\beta}}{\partial x_{\alpha}}) \frac{\partial }{\partial x_{\beta}}\bigr) I_{n-1}  & 0 \\
        0 & \frac{\mu}{\lambda+2\mu}\bigl(g^{\alpha\beta}\Gamma^\gamma_{\alpha\gamma}+\frac{\partial g^{\alpha\beta}}{\partial x_{\alpha}}\bigr)\frac{\partial }{\partial x_{\beta}}
    \end{bmatrix}\\[2mm]
    &\quad + \frac{\lambda+\mu}{\mu}
    \begin{bmatrix}
        \bigl[g^{\alpha\gamma}\Gamma^\rho_{\rho\beta}\frac{\partial }{\partial x_{\gamma}}\bigr] & \bigl[g^{\alpha\gamma}\Gamma^\rho_{\rho n}\frac{\partial }{\partial x_{\gamma}}\bigr] \\
        0 & 0
    \end{bmatrix}
    +2
    \begin{bmatrix}
        \bigl[g^{\gamma\rho}\Gamma^\alpha_{\rho\beta}\frac{\partial }{\partial x_{\gamma}}\bigr] & \bigl[g^{\gamma\rho}\Gamma^\alpha_{\rho n}\frac{\partial }{\partial x_{\gamma}}\bigr] \\[2mm]
        \frac{\mu}{\lambda+2\mu}\bigl[g^{\gamma\rho}\Gamma^n_{\rho\beta}\frac{\partial }{\partial x_{\gamma}}\bigr] & 0
    \end{bmatrix}\\[2mm]
    &\quad +
    \begin{bmatrix}
        \frac{1}{\mu}(\nabla^\alpha \mu \frac{\partial }{\partial x_{\alpha}}) I_{n-1} + \frac{1}{\mu} \big[\nabla^\alpha \lambda \frac{\partial }{\partial x_{\beta}} + g^{\alpha\gamma} \frac{\partial \mu}{\partial x_{\beta}} \frac{\partial }{\partial x_{\gamma}}\big] & \frac{1}{\mu} \frac{\partial \mu}{\partial x_{n}}\big[g^{\alpha\beta} \frac{\partial }{\partial x_{\beta}}\big] \\[2mm]
        \frac{1}{\lambda+2\mu} \frac{\partial \lambda}{\partial x_{n}} \big[\frac{\partial }{\partial x_{\beta}}\big] & \frac{1}{\lambda+2\mu} \nabla^{\alpha}\mu \frac{\partial }{\partial x_{\alpha}}
    \end{bmatrix},\\[2mm]
    C_0 &=(\lambda+\mu)
    \begin{bmatrix}
        \frac{1}{\mu}\bigl[g^{\alpha\gamma}\frac{\partial \Gamma^\rho_{\rho\beta}}{\partial x_\gamma}\bigr] & \frac{1}{\mu}\bigl[g^{\alpha\gamma}\frac{\partial \Gamma^\rho_{\rho n}}{\partial x_\gamma}\bigr] \\[2mm]
        \frac{1}{\lambda+2\mu}\bigl[\frac{\partial \Gamma^\alpha_{\alpha\beta}}{\partial x_n}\bigr] & \frac{1}{\lambda+2\mu}\frac{\partial \Gamma^\alpha_{\alpha n}}{\partial x_n}
    \end{bmatrix}
    +
    \begin{bmatrix}
        \bigl[g^{ml}\frac{\partial \Gamma^\alpha_{ml}}{\partial x_\beta}\bigr] & \bigl[g^{ml}\frac{\partial \Gamma^\alpha_{ml}}{\partial x_n}\bigr] \\[2mm]
        \frac{\mu}{\lambda+2\mu}\bigl[g^{ml}\frac{\partial \Gamma^n_{ml}}{\partial x_\beta}\bigr] & \frac{\mu}{\lambda+2\mu}g^{ml}\frac{\partial \Gamma^n_{ml}}{\partial x_n}
    \end{bmatrix}\\[2mm]
    &\quad +
    \begin{bmatrix}
        \frac{1}{\mu} \big[\nabla^\alpha \lambda \Gamma^\gamma_{\beta\gamma} - \frac{\partial \mu}{\partial x_{\gamma}} \frac{\partial g^{\alpha\gamma}}{\partial x_{\beta}}\big] & \frac{1}{\mu} \big[\nabla^\alpha \lambda \Gamma^\beta_{\beta n} - \frac{\partial \mu}{\partial x_{\beta}} \frac{\partial g^{\alpha\beta}}{\partial x_{n}}\big] \\[2mm]
        \frac{1}{\lambda+2\mu} \frac{\partial \lambda}{\partial x_{n}} [\Gamma^\alpha_{\alpha\beta}] & \frac{1}{\lambda+2\mu} \frac{\partial \lambda}{\partial x_{n}} \Gamma^\alpha_{\alpha n}
    \end{bmatrix}.
\end{align*}

\addvspace{5mm}

We then derive the microlocal factorization of the Lam\'{e} operator $\mathcal{L}_{\lambda,\mu}$.
\begin{proposition}\label{prop3.1}
    There exists a pseudodifferential operator $Q(x,\partial_{x^\prime})$ of order one in $x^\prime$ depending smoothly on $x_n$ such that
    \begin{align*}
        A^{-1}\mathcal{L}_{\lambda,\mu}
        = \Bigl(I_{n}\frac{\partial }{\partial x_n} + B - Q\Bigr)\Bigl(I_{n}\frac{\partial }{\partial x_n} + Q\Bigr)
    \end{align*}
    modulo a smoothing operator. Moreover, let $q(x,\xi^{\prime}) \sim \sum_{j\leqslant 1} q_j(x,\xi^{\prime})$ be the full symbol of $Q(x,\partial_{x^\prime})$. Then
    \begin{align}
        q_1&=|\xi^{\prime}|I_n + \frac{\lambda+\mu}{\lambda+3\mu}F_1, \label{3.9}\\
        q_{-m-1}&=\frac{1}{2|\xi^{\prime}|}E_{-m} - \frac{\lambda+\mu}{4(\lambda+3\mu)|\xi^{\prime}|^2}(F_2E_{-m}+E_{-m}F_1)\label{3.1.1}\\ 
        &\quad - \frac{(\lambda+\mu)^2}{4(\lambda+3\mu)^2|\xi^{\prime}|^3}F_2E_{-m}F_1, \quad m\geqslant -1, \notag
    \end{align}
    where
    \begin{align*}
        F_1&=
        \begin{bmatrix}
            \frac{1}{|\xi^{\prime}|}[\xi^\alpha\xi_\beta] & i[\xi^\alpha] \\[2mm]
            i[\xi_\beta] & -|\xi^{\prime}|
        \end{bmatrix},\\
        F_2&=
        \begin{bmatrix}
            \frac{1}{|\xi^{\prime}|}[\xi^\alpha\xi_\beta] & -\frac{i(\lambda+2\mu)}{\mu}[\xi^\alpha] \\[2mm]
            -\frac{i\mu}{\lambda+2\mu}[\xi_\beta] & -|\xi^{\prime}|
        \end{bmatrix},
    \end{align*}
    $E_1,E_0$ and $E_{-m}\,(m\geqslant 1)$ are given by \eqref{2.11}, \eqref{3.05} and \eqref{4.1}, respectively.
\end{proposition}

\addvspace{5mm}

\begin{proof}
    It follows from \eqref{3.07} that
    \begin{align*}
        I_{n}\frac{\partial^2 }{\partial x_n^2} + B \frac{\partial }{\partial x_n} + C 
        = \Bigl(I_{n}\frac{\partial }{\partial x_n} + B - Q\Bigr)\Bigl(I_{n}\frac{\partial }{\partial x_n} + Q\Bigr).
    \end{align*}
    Equivalently,
    \begin{align}\label{3.7}
        Q^2 - BQ - \Bigl[I_{n}\frac{\partial }{\partial x_n},Q\Bigr] + C = 0,
    \end{align}
    where the commutator $\big[I_{n}\frac{\partial }{\partial x_n},Q\big]$ is defined by, for any $v\in C^{\infty}(M)$,
    \begin{align*}
        \Bigl[I_{n}\frac{\partial }{\partial x_n},Q\Bigr]v
        &= I_{n}\frac{\partial }{\partial x_n}(Qv) - Q \Bigl(I_{n}\frac{\partial }{\partial x_n}\Bigr)v \\
        &= \frac{\partial Q}{\partial x_n}v.
    \end{align*}

    Let $q = q(x,\xi^{\prime})$ be the full symbol of the operator $Q(x,\partial_{x^\prime})$, we write $q(x,\xi^{\prime}) \sim \sum_{j\leqslant 1} q_j(x,\xi^{\prime})$ with $q_j(x,\xi^{\prime})$ homogeneous of degree $j$ in $\xi^{\prime}$. Let $b(x,\xi^{\prime})=b_1(x,\xi^{\prime}) + b_0(x,\xi^{\prime})$ and $c(x,\xi^{\prime}) = c_2(x,\xi^{\prime}) + c_1(x,\xi^{\prime}) + c_0(x,\xi^{\prime})$ be the full symbols of $B$ and $C$, respectively. Thus, $b_0=B_0$, $c_0 =C_0$ and
    \begin{align*}
        &b_1=i(\lambda+\mu)
        \begin{bmatrix}
            0 & \frac{1}{\mu}[\xi^\alpha] \\
            \frac{1}{\lambda+2\mu}[\xi_\beta] & 0
        \end{bmatrix},\\[2mm]
        &c_2= -
        \begin{bmatrix}
            |\xi^{\prime}|^2 I_{n-1} + \frac{\lambda+\mu}{\mu}[\xi^\alpha\xi_\beta] & 0 \\
            0 & \frac{\mu}{\lambda+2\mu}|\xi^{\prime}|^2
        \end{bmatrix},\\[2mm]
        &c_1= i
        \begin{bmatrix}
            \bigl(\xi^\alpha\Gamma^\beta_{\alpha\beta}+\frac{\partial \xi^\alpha}{\partial x_{\alpha}}\bigr) I_{n-1} & 0 \\
            0 & \frac{\mu}{\lambda+2\mu}\bigl(\xi^\alpha\Gamma^\beta_{\alpha\beta}+\frac{\partial \xi^\alpha}{\partial x_{\alpha}}\bigr)
        \end{bmatrix}\\[2mm]
        &\qquad +\frac{i(\lambda+\mu)}{\mu}
        \begin{bmatrix}
            [\xi^\alpha\Gamma^\gamma_{\gamma\beta}] & \Gamma^\beta_{\beta n}[\xi^\alpha] \\
            0 & 0
        \end{bmatrix} + 2i
        \begin{bmatrix}
            [\xi^\gamma\Gamma^\alpha_{\gamma\beta}] & [\xi^\gamma\Gamma^\alpha_{\gamma n}] \\[2mm]
            \frac{\mu}{\lambda+2\mu}[\xi^\gamma\Gamma^n_{\gamma\beta}] & 0 
        \end{bmatrix}\\[2mm]
        &\qquad + i
        \begin{bmatrix}
            \frac{1}{\mu} (\xi_{\alpha} \nabla^\alpha \mu) I_{n-1} + \frac{1}{\mu} \big[\xi_{\beta}\nabla^\alpha \lambda  + \xi^\alpha \frac{\partial \mu}{\partial x_{\beta}}\big]& \frac{1}{\mu} \frac{\partial \mu}{\partial x_{n}} [\xi^\alpha] \\[2mm]
            \frac{1}{\lambda+2\mu} \frac{\partial \lambda}{\partial x_{n}} [\xi_{\beta}] & \frac{1}{\lambda+2\mu} \xi_{\alpha} \nabla^{\alpha}\mu
        \end{bmatrix}.
    \end{align*} 
    
    \addvspace{3mm}

    Hence, we get the full symbol equation of \eqref{3.7}
    \begin{equation}\label{3.8}
        \sum_{J^{\prime}} \frac{(-i)^{|J^{\prime}|}}{J^{\prime} !} \partial_{\xi^{\prime}}^{J^{\prime}}q \, \partial_{x^\prime}^{J^{\prime}}q - \sum_{J^{\prime}} \frac{(-i)^{|J^{\prime}|}}{J^{\prime} !} \partial_{\xi^{\prime}}^{J^{\prime}}b \, \partial_{x^\prime}^{J^{\prime}}q - \frac{\partial q}{\partial x_n} + c = 0,
    \end{equation}
    where the sum is over all multi-indices $J^{\prime}$. 

    We shall determine $q_j$ recursively so that \eqref{3.8} holds modulo $S^{-\infty}$. Grouping the homogeneous terms of degree two in \eqref{3.8}, one has
    \begin{align*}
        q_1^2-b_1q_1+c_2=0.
    \end{align*}
    By solving the above matrix equation we get the explicit expression \eqref{3.9} for the principal symbol $q_1$ of $Q$. Here we choose that $q_1$ is positive-definite (cf. \cite{Liu19}).
    
    Grouping the homogeneous terms of degree one in \eqref{3.8}, we get the following Sylvester equation:
    \begin{align}\label{2.10}
        (q_1-b_1)q_0+q_0q_1=E_1,
    \end{align}
    where
    \begin{align}\label{2.11}
        E_1=i\sum_\alpha\frac{\partial (q_1-b_1)}{\partial \xi_\alpha}\frac{\partial q_1}{\partial x_\alpha}+b_0q_1+\frac{\partial q_1}{\partial x_n} - c_1.
    \end{align}

    Grouping the homogeneous terms of degree zero in \eqref{3.8}, we get
    \begin{align}\label{3.04}
        (q_1-b_1)q_{-1}+q_{-1}q_1=E_0,
    \end{align}
    where
    \begin{align}\label{3.05}
        E_0&=i\sum_\alpha\Bigl(\frac{\partial (q_1-b_1)}{\partial \xi_\alpha}\frac{\partial q_0}{\partial x_\alpha}+\frac{\partial q_0}{\partial \xi_\alpha}\frac{\partial q_1}{\partial x_\alpha}\Bigr)+\frac{1}{2}\sum_{\alpha,\beta}\frac{\partial^2q_1}{\partial \xi_\alpha \partial\xi_\beta}\frac{\partial^2q_1}{\partial x_\alpha \partial x_\beta} \\
        &\quad -q_0^2 +b_0q_0 +\frac{\partial q_0}{\partial x_n} - c_0.\notag
    \end{align}
    
    Proceeding recursively, grouping the homogeneous terms of degree $-m\ (m\geqslant 1)$ in \eqref{3.8}, we get
    \begin{align}\label{4.2}
        (q_1-b_1)q_{-m-1}+q_{-m-1}q_1=E_{-m},
    \end{align}
    where
    \begin{align}\label{4.1}
        E_{-m}= b_0q_{-m}+\frac{\partial q_{-m}}{\partial x_n} - i\sum_\alpha\frac{\partial b_1}{\partial \xi_\alpha}\frac{\partial q_{-m}}{\partial x_\alpha} - \sum_{\substack{-m \leqslant j,k \leqslant 1 \\ |J^{\prime}| = j + k + m}} \frac{(-i)^{|J^{\prime}|}}{J^{\prime} !} \partial_{\xi^{\prime}}^{J^{\prime}} q_j\, \partial_{x^\prime}^{J^{\prime}} q_k,
    \end{align}
    for $m \geqslant 1$. Using the method established in \cite{Liu19} we solve equations \eqref{2.10}, \eqref{3.04} and \eqref{4.2} to obtain $q_{-m-1}$ for $m \geqslant -1$, see \eqref{3.1.1}.
\end{proof}

\addvspace{5mm}

From the above Proposition \ref{prop3.1} we get the full symbol of the pseudodifferential operator $Q$. This implies that we obtain $Q$ on the boundary modulo a smoothing operator.
\begin{proposition}
    In boundary normal coordinates, the elastic Dirichlet-to-Neumann map $\Lambda_{\lambda,\mu}$ can be represented as
    \begin{align}\label{3.10}
        \Lambda_{\lambda,\mu} = AQ-D
    \end{align}
    modulo a smoothing operator, where $A$ and $D$ are given by \eqref{3.2} and \eqref{3.21}, respectively.
\end{proposition}

\addvspace{5mm}

\begin{proof}
    We use the boundary normal coordinates $(x^\prime,x_n)$ with $x_n\in[0,T]$. Since the principal symbol of the Lam\'{e} operator $\mathcal{L}_{\lambda,\mu}$ is negative-definite, the hyperplane ${x_n = 0}$ is non-characteristic, hence $\mathcal{L}_{\lambda,\mu}$ is partially hypoelliptic with respect to this boundary (see \cite[p.\,107]{Hormander64}). Therefore, the solution to the equation $\mathcal{L}_{\lambda,\mu} \textbf{\textit{u}} = 0$ is smooth in normal variable, that is, $\textbf{\textit{u}}\in [C^\infty([0,T];\mathfrak{D}^\prime (\mathbb{R}^{n-1}))]^{n}$ locally. From Proposition \ref{prop3.1}, we see that \eqref{1.3} is locally equivalent to the following system of equations for $\textbf{\textit{u}},\textbf{\textit{w}}\in [C^\infty([0,T];\mathfrak{D}^\prime (\mathbb{R}^{n-1}))]^{n}$:
    \begin{align*}
        \Bigl(I_{n+1}\frac{\partial }{\partial x_n} + Q\Bigr)\textbf{\textit{u}} & = \textbf{\textit{w}}, \quad \textbf{\textit{u}}\big|_{x_n=0}=\textbf{\textit{f}},\\
        \Bigl(I_{n+1}\frac{\partial }{\partial x_n} + B - Q\Bigr)\textbf{\textit{w}} & =\textbf{\textit{y}} \in [C^\infty([0,T];\mathfrak{D}^\prime (\mathbb{R}^{n-1}))]^{n}.
    \end{align*}
    Inspired by \cite{Liu19} (cf. \cite{LeeUhlm89}), if we substitute $t = T-x_n$ for the second equation above, then we get a backwards generalized heat equation
    \begin{align*}
        \frac{\partial \textbf{\textit{w}}}{\partial t} -(B-Q)\textbf{\textit{w}}=-\textbf{\textit{y}}.
    \end{align*}
    Since $\textbf{\textit{u}}$ is smooth in the interior of the manifold $M$ by interior regularity for elliptic operator $\mathcal{L}_{\lambda,\mu}$, it follows that $\textbf{\textit{w}}$ is also smooth in the interior of $M$, and so $\textbf{\textit{w}}|_{x_n=T}$ is smooth. In view of the principal symbol of $Q$ is positive-definite, we get that the solution operator for this heat equation is smooth for $t > 0$ (see \cite[p.\,134]{Treves80}). Therefore
    \begin{align*}
        \frac{\partial \textbf{\textit{u}}}{\partial x_n} + Q\textbf{\textit{u}}=\textbf{\textit{w}}
    \end{align*}
    locally. If we set $\mathcal{R} \textbf{\textit{f}} = \textbf{\textit{w}}|_{\partial M}$, then $\mathcal{R}$ is a smoothing operator and
    \begin{align*}
        \frac{\partial \textbf{\textit{u}}}{\partial x_n}\bigg|_{\partial M}=-Q\textbf{\textit{u}}|_{\partial M}+\mathcal{R}\textbf{\textit{f}}.
    \end{align*}
    Combining this and \eqref{3.1}, we immediately obtain \eqref{3.10}.
\end{proof}

\addvspace{10mm}

\section{Determining Lam\'{e} coefficients on the boundary}\label{s3}

\addvspace{5mm}

In this section we will prove the uniqueness results for the Lam\'{e} coefficients on the boundary by the full symbol of the elastic Dirichlet-to-Neumann map $\Lambda_{\lambda,\mu}$. We first prove Theorem \ref{thm1.3}.

\begin{proof}[Proof of Theorem {\rm \ref{thm1.3}}]
    Let $\sigma(\Lambda_{\lambda,\mu}) \sim \sum_{j\leqslant 1} p_j(x,\xi^{\prime})$ be the full symbol of the elastic Dirichlet-to-Neumann map $\Lambda_{\lambda,\mu}$. According to \eqref{3.10} and \eqref{3.21} we have
\begin{align}
    p_1&=Aq_1-d_1,\label{3.01}\\
    p_0&=Aq_0-d_0,\label{3.02}\\
    p_{-m}&=Aq_{-m},\quad m\geqslant 1,\label{3.03}
\end{align}
where
\begin{align*}
    d_1&=
    \begin{bmatrix}
        0 & i\mu[\xi^\alpha] \\
        i\lambda [\xi_\beta] & 0 
    \end{bmatrix},\\
    d_0&=
    \begin{bmatrix}
        0 & 0 \\
        \lambda [\Gamma^\alpha_{\alpha\beta}] & \lambda \Gamma^\alpha_{\alpha n} 
    \end{bmatrix},
\end{align*}
and $A$ is given by \eqref{3.2}. Therefore, it is easy to obtain \eqref{18}--\eqref{20}.
\end{proof}

\addvspace{5mm}

\begin{remark}
    For a bounded Euclidean domain $\Omega\subset\mathbb{R}^n$, Zhang in {\rm \cite{ZhangY20}} studied the elastic wave system and gave the principal symbol of the elastic Dirichlet-to-Neumann map for $n=3$. For $n\geqslant 2$, Vodev {\rm \cite{Vodev22}} showed that the elastic Dirichlet-to-Neumann map can be approximated by a pseudodifferential operator on the boundary with a matrix-valued symbol and computed the principal symbol modulo conjugation by unitary matrices.
\end{remark}

\addvspace{5mm}

We then prove the uniqueness result of Lam\'{e} coefficients on the boundary.
\begin{proof}[Proof of Theorem {\rm \ref{thm1.1}}]
Since we have obtained the explicit expression for the principal symbol $p_1$ of the elastic Dirichlet-to-Neumann map $\Lambda_{\lambda,\mu}$ in Theorem \ref{thm1.3}. It follows from \eqref{18} that the $(n,\beta)$-entry $(p_1)_\beta$ and the $(n,n)$-entry $(p_1)^{n}_{n}$ of $p_1$ are, respectively,
\begin{align*}
    (p_1)_\beta &= i f_1 \xi_\beta,\\
    (p_1)^{n}_{n} &= f_2 |\xi^{\prime}|,
\end{align*}
where
\begin{align*}
        f_1&=\frac{2\mu^2}{\lambda+3\mu},\\
        f_2&=\frac{2\mu(\lambda+2\mu)}{\lambda+3\mu}.
\end{align*}
It is easy to calculate that
\begin{align*}
    \mu&=\frac{1}{2}(f_1+f_2),\\
    \lambda&=\frac{1}{2}(f_1+f_2)\Bigl(\frac{f_2}{f_1}-2\Bigr).
\end{align*}
This shows that $p_1$ uniquely determines $\lambda$ and $\mu$ on the boundary. Furthermore, the tangential derivatives $\frac{\partial \lambda}{\partial x_\alpha}$ and $\frac{\partial \mu}{\partial x_\alpha}$ for $1\leqslant \alpha \leqslant n-1$ can also be uniquely determined by $p_1$ on the boundary.

According to the above discussion, we see from \eqref{3.02} that $q_0$ is uniquely determined by $p_0$ since the boundary values of $\lambda$ and $\mu$ have been uniquely determined. By \eqref{2.10} we can determine $E_1$ from the knowledge of $q_0$. For $k\geqslant 0$, we denote by $T_{-k}=T_{-k}(\lambda,\mu)$ the terms which involve only the boundary values of $\lambda$, $\mu$ and their normal derivatives of order ar most $k$. Note that $T_{-k}$ may be different in different expressions. From \eqref{2.11}, we have
\begin{align}\label{3.06}
    E_1=b_0q_1+\frac{\partial q_1}{\partial x_n}-c_1+T_0.
\end{align}
We calculate the $(\alpha,n)$-entry $(E_1)^\alpha$ and the $(n,n)$-entry $(E_1)^{n}_{n}$ of $E_1$,
\begin{align}
    \label{4.3} (E_1)^\alpha &= if_3 \xi^\alpha + T_0,\\
    \label{4.31} (E_1)^{n}_{n} &= f_4|\xi^{\prime}| + T_0,
\end{align}
where
\begin{align*}
    f_3 &= \frac{2}{(\lambda+3\mu)^2} \biggl( \mu\frac{\partial \lambda}{\partial x_n} -(2\lambda+3\mu)\frac{\partial \mu}{\partial x_n} \biggr),\\
    f_4 &= \frac{2}{(\lambda+2\mu)(\lambda+3\mu)^2} \biggl( \mu^2\frac{\partial \lambda}{\partial x_n} + (\lambda^2+4\lambda\mu+6\mu^2)\frac{\partial \mu}{\partial x_n} \biggr).
\end{align*}
Since $\mu>0$ and $\lambda+\mu\geqslant 0$, we have
\begin{align*}
    \det 
    \begin{bmatrix}
        \mu & -(2\lambda+3\mu) \\[1mm]
        \mu^2 & \lambda^2+4\lambda\mu+6\mu^2
    \end{bmatrix}
    =\mu(\lambda+3\mu)^2 \neq 0.
\end{align*}
This implies that $p_0$ uniquely determines $\frac{\partial \lambda}{\partial x_n}$ and $\frac{\partial \mu}{\partial x_n}$ on the boundary.

By \eqref{3.03} and \eqref{3.04} we know that $q_{-1}$ is uniquely determined by $p_{-1}$, and $E_0$ can be determined from the knowledge of $q_{-1}$. From \eqref{3.05} we see that
\begin{align*}
    E_0=\frac{\partial q_0}{\partial x_n}+T_{-1},
\end{align*}
By \eqref{2.10} we have
\begin{align*}
    (q_1-b_1)\frac{\partial q_0}{\partial x_n}+\frac{\partial q_0}{\partial x_n}q_1=\frac{\partial E_1}{\partial x_n}+T_{-1}.
\end{align*}
This implies that $\frac{\partial E_1}{\partial x_n}$ can be determined from the knowledge of $\frac{\partial q_0}{\partial x_n}$. Thus it follows from \eqref{4.3} and \eqref{4.31} that the $(\alpha,n)$-entry $\big(\frac{\partial E_1}{\partial x_n}\big)^\alpha$ and the $(n,n)$-entry $\big(\frac{\partial E_1}{\partial x_n}\big)^{n}_{n}$ of $\frac{\partial E_1}{\partial x_n}$ are, respectively,
\begin{align*}
    \Big(\frac{\partial E_1}{\partial x_n}\Big)^\alpha &= i \frac{\partial f_3}{\partial x_n} \xi^\alpha + T_{-1},\\[1mm]
    \Big(\frac{\partial E_1}{\partial x_n}\Big)^{n}_{n} &= \frac{\partial f_4}{\partial x_n}|\xi^{\prime}| + T_{-1},
\end{align*}
where
\begin{align*}
    \frac{\partial f_3}{\partial x_n} &= \frac{2}{(\lambda+3\mu)^2} \biggl( \mu\frac{\partial^2 \lambda}{\partial x_n^2} -(2\lambda+3\mu)\frac{\partial^2 \mu}{\partial x_n^2} \biggr)+T_{-1},\\
    \frac{\partial f_4}{\partial x_n} &= \frac{2}{(\lambda+2\mu)(\lambda+3\mu)^2} \biggl( \mu^2\frac{\partial^2 \lambda}{\partial x_n^2} + (\lambda^2+4\lambda\mu+6\mu^2)\frac{\partial^2 \mu}{\partial x_n^2} \biggr)+T_{-1}.
\end{align*}
Similarly, this implies that $p_{-1}$ uniquely determines $\frac{\partial^2 \lambda}{\partial x_n^2}$ and $\frac{\partial^2 \mu}{\partial x_n^2}$ on the boundary.

Finally, we consider $p_{-m-1}$ for $m\geqslant 1$. By \eqref{3.03} and \eqref{4.2} we have $p_{-m-1}$ uniquely determines $q_{-m-1}$, and $E_{-m}$ can be determined from the knowledge of $q_{-m-1}$. From \eqref{4.1} we see that
\begin{align*}
    E_{-m}=\frac{\partial q_{-m}}{\partial x_n}+T_{-m-1},
\end{align*}
We see from \eqref{4.2} that
\begin{align*}
    (q_1-b_1)\frac{\partial q_{-m}}{\partial x_n}+\frac{\partial q_{-m}}{\partial x_n}q_1=\frac{\partial E_{-m+1}}{\partial x_n}+T_{-m-1}.
\end{align*}
This implies that $\frac{\partial E_{-m+1}}{\partial x_n}$ can be determined from the knowledge of $\frac{\partial q_{-m}}{\partial x_n}$. 

We end this proof by induction. For the $(\alpha,n)$-entry $(\frac{\partial^{j} E_1}{\partial x_n^{j}})^\alpha$ and the $(n,n)$-entry $(\frac{\partial^{j} E_1}{\partial x_n^{j}})^{n}_{n}$, $1\leqslant j \leqslant m$, suppose we have shown that
\begin{align*}
    \Big(\frac{\partial^{j} E_1}{\partial x_n^{j}}\Big)^\alpha &= i \frac{\partial^{j} f_3}{\partial x_n^{j}} \xi^\alpha + T_{-j},\\[1mm]
    \Big(\frac{\partial^{j} E_1}{\partial x_n^{j}}\Big)^{n}_{n} &= \frac{\partial^{j} f_4}{\partial x_n^{j}}|\xi^{\prime}| + T_{-j},
\end{align*}
where
\begin{align*}
    \frac{\partial^{j} f_3}{\partial x_n^{j}} &= \frac{2}{(\lambda+3\mu)^2} \biggl( \mu\frac{\partial^{j+1} \lambda}{\partial x_n^{j+1}} -(2\lambda+3\mu)\frac{\partial^{j+1} \mu}{\partial x_n^{j+1}} \biggr)+ T_{-j},\\[1mm]
    \frac{\partial^{j} f_4}{\partial x_n^{j}} &= \frac{2}{(\lambda+2\mu)(\lambda+3\mu)^2} \biggl( \mu^2\frac{\partial^{j+1} \lambda}{\partial x_n^{j+1}} + (\lambda^2+4\lambda\mu+6\mu^2)\frac{\partial^{j+1} \mu}{\partial x_n^{j+1}} \biggr)+ T_{-j}.
\end{align*} 
By \eqref{3.03} and \eqref{4.2} we see that $p_{-j}$ uniquely determines $q_{-j}$, and $E_{-j+1}$ can be determined from the knowledge of $q_{-j}$. By iteration, we have $\frac{\partial q_{-m}}{\partial x_n}$ can be determined from the knowledge of $\frac{\partial^{m+1} E_1}{\partial x_n^{m+1}}$. Then, the $(\alpha,n)$-entry $(\frac{\partial^{m+1} E_1}{\partial x_n^{m+1}})^\alpha$ and the $(n,n)$-entry $(\frac{\partial^{m+1} E_1}{\partial x_n^{m+1}})^{n}_{n}$ are, respectively,
\begin{align*}
    \Big(\frac{\partial^{m+1} E_1}{\partial x_n^{m+1}}\Big)^\alpha &= i \frac{\partial^{m+1} f_3}{\partial x_n^{m+1}} \xi^\alpha + T_{-m-1},\\[1mm]
    \Big(\frac{\partial^{m+1} E_1}{\partial x_n^{m+1}}\Big)^{n}_{n} &= \frac{\partial^{m+1} f_4}{\partial x_n^{m+1}}|\xi^{\prime}| + T_{-m-1},
\end{align*}
where
\begin{align*}
    \frac{\partial^{m+1} f_3}{\partial x_n^{m+1}} &= \frac{2}{(\lambda+3\mu)^2} \biggl( \mu\frac{\partial^{m+2} \lambda}{\partial x_n^{m+2}} -(2\lambda+3\mu)\frac{\partial^{m+2} \mu}{\partial x_n^{m+2}} \biggr)+ T_{-m-1},\\[1mm]
    \frac{\partial^{m+1} f_4}{\partial x_n^{m+1}} &= \frac{2}{(\lambda+2\mu)(\lambda+3\mu)^2} \biggl( \mu^2\frac{\partial^{m+2} \lambda}{\partial x_n^{m+2}} + (\lambda^2+4\lambda\mu+6\mu^2)\frac{\partial^{m+2} \mu}{\partial x_n^{m+2}} \biggr)+ T_{-m-1}.
\end{align*}
By the same argument, we have $p_{-m-1}$ uniquely determines $\frac{\partial^{m+2} \lambda}{\partial x_n^{m+2}}$ and $\frac{\partial^{m+2} \mu}{\partial x_n^{m+2}}$ on the boundary. Therefore, we conclude that the elastic Dirichlet-to-Neumann map $\Lambda_{\lambda,\mu}$ uniquely determines $\frac{\partial^{|J|} \lambda}{\partial x^J}$ and $\frac{\partial^{|J|} \mu}{\partial x^J}$ on the boundary for all multi-indices $J$.
\end{proof}

\addvspace{10mm}

\section{Global uniqueness of real analytic Lam\'{e} coefficients}\label{s4}

\addvspace{5mm}

This section is devoted to proving the global uniqueness of real analytic Lam\'{e} coefficients on a real analytic manifold. More precisely, we prove that the elastic Dirichlet-to-Neumann map $\Lambda_{\lambda,\mu}$ uniquely determines the real analytic Lam\'{e} coefficients on the whole manifold $\bar{M}$. 

We recall that the definitions of real analytic functions and real analytic hypersurfaces of a Riemannian manifold. Let $f(x)$ be a real-valued function defined on an open set $\Omega\subset\mathbb{R}^n$. For $y\in\Omega$ we call $f(x)$ real analytic at $y$ if there exist $a_J \in\mathbb{R}$ and a neighborhood $N_y$ of $y$ such that
\begin{align*}
    f(x)=\sum_{J} a_J(x-y)^J
\end{align*}
for all $x\in N_y$ and $J\in\mathbb{N}^n$. We say $f(x)$ is real analytic on an open set $\Omega$ if $f(x)$ is real analytic at each $y\in\Omega$.

Let $(M,g)$ be a Riemannian manifold. A subset $U$ of $M$ is said to be an $(n-1)$-dimensional real analytic hypersurface if $U$ is nonempty and if for every point $x\in U$, there is a real analytic diffeomorphism of an unit open ball $B(0,1)\subset\mathbb{R}^n$ onto an open neighborhood $N_x$ of $x$ such that $B(0,1)\cap\{x\in\mathbb{R}^n|x_n = 0\}$ maps onto $N_x\cap U$.

\addvspace{3mm}

In order to prove Theorem \ref{thm1.2}, we need the following lemma (see \cite[p.\,65]{John82}).
\begin{lemma}\label{lem3.1}
    $($Unique continuation of real analytic functions$)$ Let $\Omega\subset\mathbb{R}^n$ be a connected open set and $f(x)$ be a real analytic function defined on $\Omega$. Let $y\in\Omega$. Then $f(x)$ is uniquely determined in $\Omega$ if we know $\frac{\partial^{|J|}f(y)}{\partial x^J}$ for all $J\in\mathbb{N}^n$. In particular, $f(x)$ is uniquely determined in $\Omega$ by its values in any nonempty open subset of $\Omega$.
\end{lemma}

Note that the above lemma still holds for real analytic functions defined on real analytic manifolds. Finally, we prove Theorem \ref{thm1.2}.

\addvspace{5mm}

\begin{proof}[Proof of Theorem {\rm \ref{thm1.2}}]
    According to Theorem \ref{thm1.1}, it has been proved that the elastic Dirichlet-to-Neumann map $\Lambda_{\lambda,\mu}$ uniquely determines $\frac{\partial^{|J|} \lambda}{\partial x^J}$ and $\frac{\partial^{|J|} \mu}{\partial x^J}$ on the boundary for all multi-indices $J$. Hence, for any point $x_0\in\Gamma$, the Lam\'{e} coefficients can be uniquely determined in some neighborhood of $x_0$ by the analyticity of the Lam\'{e} coefficients on $M\cup\Gamma$. Furthermore, it follows from Lemma \ref{lem3.1} that the Lam\'{e} coefficients can be uniquely determined in $M$. Therefore, by combining Theorem \ref{thm1.1} we conclude that the Lam\'{e} coefficients can be uniquely determined on $\bar{M}$.
\end{proof}

\addvspace{5mm}

\begin{remark}
    By applying the method of Kohn and Vogelius {\rm \cite{KohnVoge85}}, we can also prove that the elastic Dirichlet-to-Neumann map $\Lambda_{\lambda,\mu}$ uniquely determines $\lambda$ and $\mu$ on $\bar{M}$ provided the manifold and the Lam\'{e} coefficients are piecewise analytic.
\end{remark}

\addvspace{10mm}

\section*{Acknowledgements}

\addvspace{5mm}

This research was supported by NSFC (12271031) and NSFC (11671033/A010802).

\addvspace{10mm}

\addvspace{5mm}

\end{document}